\documentclass[reqno]{amsart}

\usepackage{hyperref}
\usepackage[all]{xy}
\usepackage{mathrsfs}
\usepackage{amssymb}
\usepackage{stmaryrd}

\numberwithin{equation}{section}

\newtheorem{MainThm}{Theorem}

\newtheorem*{clthm}{Classical Theorem}
\newtheorem{thm}[equation]{Theorem}

\newtheorem{lem}[equation]{Lemma}
\newtheorem{prop}[equation]{Proposition}

\newtheorem{construction}[equation]{Construction}

\theoremstyle{remark}

\newtheorem{rems}[equation]{Remarks}
\newtheorem{rem}[equation]{Remark}
\newtheorem{explanations}[equation]{Explanations}
\newtheorem{example}[equation]{Example}

\theoremstyle{definition}

\newtheorem{defn}[equation]{Definition}

\newcommand{\SO}{\mathrm{SO}}

\newcommand{\bQ}{\mathbb{Q}}
\newcommand{\bR}{\mathbb{R}}
\newcommand{\bZ}{\mathbb{Z}}
\newcommand{\bC}{\mathbb{C}}

\newcommand{\Diff}{\mathrm{Diff}}
\newcommand{\GL}{\mathrm{GL}}
\newcommand{\hAut}{\mathrm{hAut}}
\newcommand{\smb}{\mathrm{smb}}

\newcommand{\Harm}{\mathcal{H}}
\newcommand{\id}{\mathrm{id}}
\newcommand{\inddiff}{\mathrm{inddiff}}

\newcommand{\coker}{\mathrm{coker}}

\newcommand{\colim}{\mathrm{colim}}
\newcommand{\hocolim}{\mathrm{hocolim}}
\newcommand{\cA}{\mathcal{A}}
\newcommand{\im}{\mathrm{im}}

\newcommand{\Eig}{\mathrm{Eig}}
\newcommand{\hur}{\mathrm{hur}}
\newcommand{\thom}{\mathrm{thom}}
\newcommand{\susp}{\mathrm{susp}}

\newcommand{\scpr}[1]{\langle #1 \rangle}
\newcommand{\bN}{\mathbb{N}}

\newcommand{\MT}{\mathrm{MT}}

\newcommand{\psc}{\mathrm{psc}}
\newcommand{\fX}{\mathfrak{X}}

\newcommand{\cp}{\mathbb{C}\mathbb{P}}
\newcommand{\ind}{\mathrm{ind}}

\newcommand{\Aut}{\mathrm{Aut}}
\newcommand{\ahat}{\hat{\mathfrak{a}}}

\newcommand{\Spin}{\mathrm{Spin}}
\newcommand{\twomatrix}[4]{\begin{pmatrix} #1 & #2 \\ #3& #4 \end{pmatrix}}

\newcommand{\Riem}{\mathcal{R}}
\newcommand{\scal}{\mathrm{scal}}
\newcommand{\bF}{\mathbb{F}}
\newcommand{\Bun}{\mathrm{Bun}}

\newcommand{\cW}{\mathcal{W}}
\newcommand{\cK}{\mathcal{K}}

\newcommand{\Homeo}{\mathrm{Homeo}}

\newcommand{\sign}{\mathrm{sign}}

\newcommand{\st}{\mathrm{st}}
\newcommand{\KO}{\mathrm{KO}}

\newcommand{\hq}{\sslash}

\newcommand{\Kerv}[1]{\mathrm{Kerv}_{#1}}

\title[Cancellation properties]{Cancellation properties for exotic $4$-dimensional positive scalar curvature metrics}

\author{Johannes Ebert}
\email{johannes.ebert@uni-muenster.de}
\address{
Mathematisches Institut\\
WWU M{\"u}nster\\
Einsteinstr. 62\\
48149 M{\"u}nster\\
Germany
}

\thanks{The author was supported by the Deutsche Forschungsgemeinschaft (DFG, German Research Foundation) -- Project-ID 427320536 -- SFB 1442, as well as under Germany’s Excellence Strategy EXC 2044 -- 390685587, Mathematics M\"unster: Dynamics–Geometry–Structure.
}

\date{\today}

\begin{document}

\begin{abstract}
Ruberman constructed families $\{g_n\vert n \in \bN\} \subset \Riem^+ (M)$ of metrics of positive scalar curvature on certain $4$-manifolds which are concordant but lie in different path components of $\Riem^+ (M)$. We prove a cancellation result along the following lines. For each closed manifold $N$, there is a map $\nu_N: \Riem^+ (M) \to \Riem^+ (M \times N)$, well-defined up to homotopy, that takes the product with $N$. We prove that when $N$ has positive dimension $\nu_N$ takes all metrics of Ruberman's family to the same path component. This is trivial when $N$ has a psc metric and follows from pseudoisotopy theory when $\dim (N) \geq 3$. 

More generally, we give easily checkable conditions on diffeomorphisms $f:M \to M$ of $1$-connected $4$-manifolds which guarantee that for each closed $N$ and each $g \in \Riem^+ (M)$, the metrics $\nu_N (g)$ and $\nu_N (f^* g)$ lie in the same path component of $\Riem^+ (M \times N)$. The conditions are (A) the cobordism class of the mapping torus $[T(f)] \in\Omega_5^\SO$ vanishes and (B) the induced map $f^* : H^2 (M;\bR) \to H^2(M;\bR)$ belongs to the unit component of $\Aut (I_M)$, the orthogonal group of the intersection form of $M$. Depending on whether $\dim (N)=1$ or $\dim(N) \geq 2$ and on whether or not $M$ is spin, various combinations of (A) and (B) are needed. 

The proof uses rigidity properties for the diffeomorphism action on $\Riem^+ (M)$ for high--dimensional $M$, which in the necessary generality were established by Frenck and Bantje. These rigidity properties reduce the problem to a calculation of the homotopy group $\pi_1 (\MT \SO(4))$ of the Madsen--Tillmann spectrum which we also carry out. 

Recently, Auckly and Ruberman exhibited examples of elements in higher homotopy groups of $\Riem^+(M^4)$ for certain $M$. Using the same method, we also prove that these elements lie in the kernel of the induced map $(\nu_N)_*$ on rational homotopy.
\end{abstract}

\maketitle

\tableofcontents

\section{Introduction}

It is well-known that the many exotic phenomena of smooth $4$-dimensional topology are very fragile under various sorts of stabilization. The purpose of this note is to show a similar behavior for some exotic $4$-dimensional positive scalar curvature phenomena. Before we turn to positive scalar curvature, let us mention a classical result from $4$-dimensional topology to get an impression of the spirit of our main results. 

\begin{clthm}
Let $M_0$ and $M_1$ be closed, $1$-connected smooth $4$-manifolds with isomorphic intersection forms. Then 
\begin{enumerate}
\item $M_0$ and $M_1$ are $h$-cobordant (Wall \cite[Theorem 2]{Wall1co4}),
\item $M_0$ and $M_1$ are not necessarily diffeomorphic, 
\item however $M_0$ and $M_1$ are homeomorphic (Freedman \cite{Freedman}),
\item we have \emph{handle stabilization}: $M_0 \sharp k(S^2 \times S^2)$ and $M_1 \sharp k(S^2 \times S^2)$ are diffeomorphic for sufficiently large $k$ (Wall \cite[Theorem 3]{Wall1co4}),
\item we have \emph{product stabilization}: if $N$ is closed of positive dimension, then $M_0 \times N \cong M_1 \times N$.
\end{enumerate}
\end{clthm}
Examples for (2) exist in abundance and can be detected by means of Seiberg--Witten invariants; see the textbook \cite{Scorpan} for an overview and a guide to the vast literature. Item (5) is folklore and has the following easy proof: the product of an $h$-cobordism $W:M_0 \leadsto M_1$ with $N$ is an $h$-cobordism $W \times N: M_0 \times N \leadsto M_1 \times N$ whose Whitehead torsion vanishes since that of $W$ vanishes, and one can apply the $s$-cobordism theorem to $W \times N$. 

Gauge theory also detects exotic phenomena related to positive scalar curvature. This concerns the existence question, as well as the topology of the space $\Riem^+ (M)$ of all such metrics. In light of the above result, one might suspect that these exotic phenomena exhibit similar cancellation properties. The purpose of this note is to prove that this is indeed the case. 

It is well-known that in the high-dimensional regime ($\dim (M) \geq 5$), the theory of positive scalar curvature is governed by cobordism theory; this is due to the Gromov--Lawson surgery theorem \cite{GL} \cite{Chernysh}; see \cite[Theorems 1.2 and 1.5]{EbertFrenck} for the most general formulation. 
In dimension $4$, we can at least infer that this is true ``stably'', in the following sense. 

\begin{MainThm}\label{thm:existence-questionstably}
Let $B \stackrel{\theta}{\to} BO$ be a fibration and let $M_0$ be a closed $4$-manifold, together with a lift $\ell:M_0 \to B$ of the stabilized tangent bundle of $M_0$, and suppose that $\ell$ is $2$-connected. 
Assume furthermore that the cobordism class $[M_0] \in \Omega_4^\theta$ contains a representative $[M_1]$ which has a psc metric. Then
\begin{enumerate}
\item For sufficiently large $k$, $M_0 \sharp k (S^2 \times S^2)$ has a psc metric (handle stabilization).
\item If $N$ is a closed manifold of positive dimension, then $M_0 \times N$ as a psc metric (product stabilization). 
\end{enumerate}
\end{MainThm}

This result is certainly known and a rather easy consequence of the Gromov--Lawson surgery theorem. We give the proof in \S \ref{sec:existence}. 
If $M_0$ had dimension $\geq 5$, the conclusion that $M_0$ has a psc metric is valid without stabilization, by \cite[Theorem 1.5]{EbertFrenck}. The latter conclusion fails in dimension $4$: 

\begin{example}
Each cobordism class in $\Omega^{\SO}_4\cong \bZ \{[\cp^2]\}$ admits a psc representative. Therefore, in order to find a counterexample, one needs a $1$-connected $M^4$ which is not spin and does not admit a psc metric. 

A $4$-dimensional K\"ahler manifold $M$ with $b_2^+(M) \geq 2$ does not have a psc metric, by a result of Taubes \cite{Taubes}, combined with e.g. \cite[Corollary 5.18]{Morgan}. 

Examples of such K\"ahler manifolds which are $1$-connected and non-spin are easy to find: a degree $d$ algebraic hypersurface $X_d \subset \cp^3$ is $1$-connected by the Lefschetz hyperplane theorem, and its characteristic numbers can be calculated by a classical method \cite{Milnor4}. The outcome is that for $d \geq 5$ is odd, $X_d$ is nonspin and has $b_2^+(X_d)\geq 2$.
\end{example}

Let us now turn, finally, to the main aim of this note, that is $\pi_0(\Riem^+ (M))$. There are examples of $4$-manifolds with ``exotic'' $\pi_0 (\Riem^+ (\_))$. For instance, the following is proven in \cite[\S 5]{Ruberman}.

\begin{thm}[Ruberman]\label{thm:ruberman}
There exists a $1$-connected $4$-manifold\footnote{For example $M= 4 (\cp^2) \sharp 21 (\overline{\cp^2})$.} $M$, a psc metric $g_0 \in \Riem^+ (M)$ and a family $\{f_i\vert i \in \bN_0\}$ of orientation-preserving diffeomorphisms $f_i:M \to M$ with $f_0=\id$, with the following properties.
\begin{enumerate}
\item Each $f_i$ is pseudoisotopic to the identity, i.e. there are diffeomorphisms $F_i: M \times [0,1] \to M \times [0,1]$ such that $F_i (x,0)=(x,0)$ and $F_i (x,1)= (f_i (x),1)$ for each $x \in M$. 
\item The metrics $g_i:= f_i^* g_0$ are pairwise nonisotopic, i.e. the elements $[g_i] \in \pi_0 (\Riem^+ (M))$ are pairwise distinct. 
\end{enumerate}
\end{thm}

Not surprisingly, that $[g_i] \neq [g_j]$ for $i \neq j$ is detected by Seiberg--Witten invariants. 
Note that the psc metrics $g_i$ are all concordant, via $F_i^* (g_0 \oplus dx^2)$. To best of the knowledge of the author, these are the only known examples of concordant but not isotopic psc metrics (index theoretic methods cannot address this question at all, see \cite{Hertl} for a high-brow variant of this statement).

These examples satisfy ``handle stabilization'': after taking the connected sum with sufficiently many copies of $S^2 \times S^2$'s, all the metrics $g_i$ become isotopic. 
More precisely, we may suppose that there is an embedded disc $D^4 \subset M$ fixed (pointwise) by each $f_i$, and we may also suppose that $g_0$ is a torpedo metric inside $D^4$. We can therefore take the connected sum of $f_i$ and the identity on $k(S^2 \times S^2)$, and likewise we can take the connected sum of $g_0$ with some suitable standard metric on $k(S^2 \times S^2)$. 
It follows from \cite[Theorem 1.4]{Quinn} (as corrected by \cite{Gabai}) that after taking connected sum with $k(S^2 \times S^2)$ for suitable $k$, $f_i$ becomes isotopic to the identity, so that $g_i$ and $g_0$ become isotopic (handle stabilization). 

Note that for $M^{2n}$ with $n \geq 3$, the connected sum map 
\[
\Riem^+ (M) \to \Riem^+ (M \sharp (S^n \times S^n))
\]
is a homotopy equivalence, by \cite[Theorem 1.2]{EbertFrenck}; the above reasoning proves that this is not true when $n=2$. 

We wish to prove that in the situation of Theorem \ref{thm:ruberman}, we also have product stabilization. Beforehand, we have have to clarify what is meant by that. 

\begin{construction}
Let $M$ be a closed manifold. For $a \in \bR$, we write $\Riem^{>a}(M)$ for the space of all metrics $g$ with $\inf_M \scal(g) >a$ (so that $\Riem^+ (M)=\Riem^{>0}(M)$).
Given $\epsilon>0$ and another closed manifold $N$, the formula $(g,h) \mapsto g \oplus h$ defines a map 
\[
\nu: \Riem^{>\epsilon}(M) \times \Riem^{>-\epsilon}(N)\to \Riem^+ (M \times N)
\]
A very easy scaling argument shows that the inclusion $\Riem^{>\epsilon}(M) \to \Riem^+ (M)$ is a homotopy equivalence, whereas $\Riem^{>-\epsilon}(N)$ is contractible by \cite[Lemma 5.4]{EbertWiemeler}. Hence we may interpret $\nu$, up to homotopy, as a map 
\begin{equation}\label{eqn:productmap}
\nu_N : \Riem^+ (M) \to \Riem^+ (M \times N),
\end{equation}
whose homotopy class is independent of $\epsilon>0$. 
\end{construction}

\begin{rem}\label{rem:numap}
The homotopical behavior of $\nu_N$ depends strongly on $N$. 
\begin{enumerate} 
\item If $N$ does admit a metric of positive scalar curvature, the map $\nu_N$ factors through some $\Riem^{-\epsilon}(M)$ and is therefore homotopic to a constant. 
\item On the other hand, $\nu_N$ behaves nicely with respect to the index difference map $\inddiff$ (assuming spin structures throughout), in the sense that ($n:= \dim (N)$, $G:= \pi_1 (N)$)
\[
\xymatrix{
\pi_k (\Riem^+ (M)) \ar[d]^{(\nu_N)_*} \ar[r]^{\inddiff} & KO_{k+d+1} (*) \ar[d]^{\_ \times \ahat_G(N)}\\
\pi_k (\Riem^+ (M \times N)) \ar[r]^{\inddiff^G} & KO_{k+d+n+1} (C^* (G)) 
}
\]
commutes; here $\ahat(N) \in KO_n (C^* (G))$ is the Rosenberg index of $N$. If e.g. $N=S^1$, the right vertical map is injective and if $d \geq 6$, the top horizontal map is very much nontrivial by the main result of \cite{BERW}, so $(\nu_N)_*$ is nontrivial. 
\end{enumerate}
\end{rem}

Let us turn to the $4$-dimensional world. In light of the construction of the metrics $g_i$ in Theorem \ref{thm:ruberman}, it seems plausible that they also exhibit product stabilization. Our main aim is to confirm that this is indeed the case: 

\begin{thm}\label{thm:killing}
Let $M$ and $g_i:= f_i^* g_0$ be as in Theorem \ref{thm:ruberman} and let $N$ be a closed connected smooth manifold of dimension $n\geq 1$. Then the metrics $\nu_N (g_i)$ all lie in the same path component of $\Riem^+ (M \times N)$. 
\end{thm}

By Remark \ref{rem:numap}, this is trivially true if $N$ has a psc metric. For $n \geq 3$, we could argue by pseudoisotopy theory (e.g. \cite[p. 11]{HatcherWagoner}) as follows. Since $\pi_1 (M)=1$, the pseudoisotopy obtruction of $f_i$ vanishes. Therefore, so does that of $f_i \times \id_N$, and if $n \geq 3$, the quoted result implies that $f_i \times \id_N$ is isotopic to the identity. In light of the construction of $g_i$, this implies Theorem \ref{thm:killing} for $\dim (N) \geq 3$ (the author was informed by Daniel Ruberman that a direct proof that $f_i \times \id_{S^1}$ is isotopic to the identity is possible, which also gives a proof of Theorem \ref{thm:killing} for $\dim (N) =1$).

The purpose of this note is to give a proof of Theorem \ref{thm:killing} that also works for $\dim (N)=1$ and $2$, is cobordism-theoretic in nature rather than through pseudoisotopy, and does not depend on an understanding how the diffeomorphisms $f_i$ are constructed. Moreover, we only need very mild homological conditions on $f_i$ and not that $f_i$ is pseudoisotopic to the identity. 

For the formulation of the result, recall that the \emph{mapping torus} $T(f)$ of $f \in \Diff (M)$ is $T(f):= (M\times [0,1])/ (x,1) \sim (f(x),0)$, which is a fibre bundle over $S^1$ with fibre $M$. If $M$ is oriented and $f$ is orientation--preserving, $T(f)$ has a natural orientation. Assigning to $f$ the cobordism class $[T(f)]\in \Omega_{d+1}^\SO$ ($d=\dim (M)$) gives a map
\begin{equation}\label{eqn:mappingtorusmap-SO}
\Gamma^+ (M^d):= \pi_0 (\Diff^+(M^d))\to \Omega_{d+1}^\SO, \; [f] \mapsto [T(f)],
\end{equation}
which is a group homomorphism. We remark that $\Omega^\SO_5 \cong \bZ/2$, via the Stiefel--Whitney number $[M] \mapsto w_2 w_3 (M)$. For mapping tori, $w_2 w_3 (T(f))$ can be computed in terms of the action of $f$ on $H^*(M;\bR)$ and $H^*(M;\bF_2)$ (see Lemma \ref{lem:kervaire-mappingtorus} below):
\begin{equation}\label{eqnw2w3mappingtorus}
w_2 w_3 ([T(f)]) = \dim (\Eig_1 (H^2 (f;\bR))) + \dim (\Eig_1 (H^2 (f;\bF_2))) \pmod 2 
\end{equation}
(we use the notation $H^k(f;R)$ for the map induced by $f$ in cohomology with coefficients in $R$, and the notation $\Eig_\lambda(F)$ for the $\lambda$-eigenspace of a linear endomorphism $F$). 
Let $I_M: H^2 (M;\bR) \times H^2 (M;\bR)\to \bR$ be the intersection form of $M$, a symmetric nondegenerate bilinear form on a finite-dimensional $\bR$-vector space. We let $\Aut (I_M) \subset \GL (H^2(M;\bR))$ be the orthogonal group of $I_M$. For $f \in \Diff^+ (M)$, the induced map $H^2(f;\bR)$ belongs to $\Aut(I_M)$. We view $\Aut (I_M)$ as a Lie group; as such it is isomorphic to the indefinite orthogonal group $O(b_2^+(M),b_2^-(M))$. It is well-known that $\pi_0 (\Aut I_M) \cong \bZ/2$ or $\cong (\bZ/2)^2$, depending on whether $I_M$ is definite or not. 
\begin{MainThm}\label{thm:rigidity-dimension4}
Let $f:M \to M$ be an orientation-preserving diffeomorphism of a $1$-connected closed smooth $4$-manifold. Let $N$ be a closed connected manifold of dimension $n \geq 1$. Then 
\[
(f \times \id)^*: \Riem^+ (M \times N) \to \Riem^+ (M \times N)
\]
is homotopic to the identity under each the following assumptions.
\begin{enumerate}
\item $n \geq 2$ and $M$ is spin;
\item $n \geq 2$, $M$ is not spin and $[T(f)] =0 \in \Omega^\SO_5$;
\item $n =1$, $M$ is spin, and $[H^2 (f;\bR)] =1 \in \pi_0(\Aut (I_M))$;
\item $n=1$, $M$ is not spin, $[T(f)]=0$ and $[H^2 (f;\bR)] =1 \in \pi_0(\Aut (I_M))$.
\end{enumerate}
\end{MainThm}


In Lemma \ref{lem:realizationofinvariants}, we give elementary examples of diffeomorphisms whose invariants in $\Omega^\SO_5$ and $\pi_0(\Aut (I_M))$ do not vanish. 

The proof of the Theorem rests on rigidity theorems for the action of diffeomorphism groups on $\Riem^+ (M)$ for high-dimensional $M$, in particular on the works by Frenck \cite{Frenck} (when $n \geq 2$) and Bantje \cite{Bantje} (when $n=1$). We give a review of these results in \S \ref{sec:rigidity}. What these results have in common is that they govern the action of $\Diff(M)$ on $\Riem^+ (M)$ in terms of cobordism theory; we give more precise statements below.
Using these results, the proof of Theorem \ref{thm:rigidity-dimension4} comes down to verifying that the mapping torus $T(f)$ vanishes in a certain cobordism group. When $n \geq 2$, the relevant cobordism group is just $\Omega_5^\SO$ or $\Omega^\Spin_5=0$. When $n =1$, the relevant cobordism group is the first homotopy group $\pi_1 (\MT \SO(4))$ or $\pi_1 (\MT \Spin(4))$ of the Madsen--Tillmann spectrum. In \S \ref{sec:calculation}, we give a detailed calculation of these groups. The answer we find accounts for the conditions of (3) and (4) of Theorem \ref{thm:rigidity-dimension4}. 

Let us now turn to the last of our main results, concerning a recent result by Auckly and Ruberman \cite[Theorem 1.14]{AucklyRuberman} that reads as follows. 
\begin{thm}[Auckly--Ruberman]\label{thm:auckrub}
For each $p\geq 0$ and $j \leq p$ with $p-j$ even, there are manifolds $M$ of the form $M=k(S^2 \times S^2)$ or $M=k(\cp^2 \sharp\overline{\cp^2} )$ for large $k$ and $g_0 \in \Riem^+ (M)$ such that 
\begin{enumerate}
\item $\ker (\pi_j (\Diff^+ (M))\to \pi_j (\Homeo^+ (M)))$ contains a free abelian group $A$ of infinite rank,
\item The map $o_*: \pi_j (\Diff^+ (M)) \to \pi_j (\Riem^+ (M))$ induced by the orbit map $o: f \mapsto f^* g_0$, is injective on the free abelian subgroup $A$ of (1).
\end{enumerate}
\end{thm}

\begin{MainThm}\label{thm:hogherhomotopy}
Let $M$ and $j$ as in Theorem \ref{thm:auckrub}. For any $N$ of positive dimension and any $j>0$, the map 
\[
(\nu_N)_*:\pi_j (\Riem^+ (M)) \to \pi_j (\Riem^*(M \times N)) \otimes \bQ
\]
vanishes on the subgroup $o_* (A)$ described in Theorem \ref{thm:auckrub}. 
\end{MainThm}
The proof is along similar lines as that of the $n=1$ case of Theorem \ref{thm:rigidity-dimension4}. This time, we have a map $\pi_j (\Diff^+ (M))\to \pi_{j+1} (\MT \SO(4))$ or $\MT\Spin (4)$. Though for $j\geq 1$, these homotopy groups are very hard to calculate, one checks easily that classes in the kernel of $\pi_j (\Diff^+(M)) \to \pi_j (\Homeo^+(M))$ are mapped to finite order elements. 

\subsection{Acknowledgements}

The idea for Theorem \ref{thm:rigidity-dimension4} was born during a talk by Hokuto Konno at the Oberwolfach workshop on positive scalar curvature in 2024; and I am grateful to him for sharing his insights with me afterwards, and also to point me to the paper \cite{AucklyRuberman}. I am also grateful to Daniel Ruberman for comments on this paper.

\section{Proof of Theorem \ref{thm:existence-questionstably}}\label{sec:existence}

\begin{proof}[Proof of Theorem \ref{thm:existence-questionstably}]
Suppose that $M_1$ is a $\theta$-manifold having a psc metric and such that $[M_0]=[M_1]\in \Omega_4^\theta$. By surgery below the middle dimension (which works in dimension $4$ without change), $[M_1]$ contains a representative $[M_2]$ such that the structure map $M_2 \to B$ is $2$-connected. 
Since we can obtain $M_2$ from $M_1$ by doing surgeries along embedded spheres of dimension $\leq 1$, i.e. of codimension $\geq 3$, the Gromov--Lawson surgery theorem implies that $M_2$ carries a psc metric as well. Theorem C of \cite{Kreck} shows that $M_2 \sharp \ell (S^2 \times S^2) \cong M_0 \sharp k (S^2 \times S^2)$ for suitable $k,\ell$. Since $M_2$ has a psc metric, so does $M_2 \sharp \ell (S^2 \times S^2) \cong M_0 \sharp k (S^2 \times S^2)$, proving claim (1). 

For claim (2), let $M_2$ be as above and note that $M_2 \times N$ carries a psc metric (by shrinking the metric on $M_2$). Let $B_N$ be the normal $2$-type of $N$. The structure maps $M_i \times N \to B \times B_N$ are $2$-connected, and when $\theta\times\theta_N: B \times B_N \to BO \times BO$ denotes the sum of the two tangential structures, we have $[M_0\times N]=[M_2 \times N] \in \Omega_{4+\dim (N)}^{\theta \times \theta_N}$. Now \cite[Theorem 1.5]{EbertFrenck} implies that $M_0 \times N$ has a psc metric.
\end{proof}

\section{Rigidity theorem for the diffeomorphism action in high dimensions}\label{sec:rigidity}

In this section, we describe the aforementioned rigidity theorems and show how they are used in the proof of Theorem \ref{thm:rigidity-dimension4}. This will lead to a cobordism-theoretic criterion concerning $f$; if $n \geq 2$, this criterion is easily checked, and the relevant calculation in the case $n=1$ is carried out in \S \ref{sec:calculation}. 

\subsection{Frenck's theorem}\label{subsec:georgstheorem}

Let us now actually state the rigidity theorems, starting with \cite{Frenck} and referring the reader to \cite{Frenck} and \S 1.7--1.8 of the survey \cite{BotvinnikEbert} for more details. 

For a space $X$, denote by $\hAut(X)$ the space of all homotopy equivalences $X \to X$. This is a topological monoid (composition of maps). Of course $\hAut(X)$ is not a topological group, but it is \emph{grouplike} in the sense that $\pi_0 (\hAut(X))$, with the composition induced by that of $\hAut(X)$, is a group.

For a closed manifold $M$, the action of the diffeomorphism group $\Diff(M)$ on $\Riem^+ (M)$ provides a homomorphism 
\[
\Diff(M) \to \hAut(\Riem^+(M))
\]
of topological monoids. On level of path components, it induces a group homomorphism 
\begin{equation}\label{eqn:actionhomomorphism}
\Gamma (M):= \pi_0 (\Diff(M)) \to \pi_0 (\hAut(\Riem^+ (M))).
\end{equation}
The source group is the group of isotopy classes of diffeomorphisms of $M$ and classically called the \emph{mapping class group} of $M$. In favorable cases, quite a bit is known about the structure of $\Gamma(M)$, but it typically is a complicated and large group.

The target group is the group of homotopy classes of homotopy self-equivalences of the (mysterious) space $\Riem^+ (M)$. We cannot expect to say anything sensible about $\pi_0 (\hAut(\Riem^+ (M)))$ unless $\Riem^+ (M)=\emptyset$ or $\Riem^+ (M)\simeq *$ in which case it is the trivial group. 

The surprising result from \cite{Frenck} is that when $\dim (M) \geq 6$, even though source and target of \eqref{eqn:actionhomomorphism} are complicated/intractable, the homomorphism \eqref{eqn:actionhomomorphism} itself is relatively easy to control in terms of cobordism theory. 

The precise statement involves the notion tangentially structured diffeomorphisms which we briefly recall. Let $\theta: B \to BO$ be a map, and let us write $\theta(n):B(n) \to BO(n)$ for the homotopy pullback along $BO(n) \to BO$ and $V^\theta_n \to B(n)$ for the pullback of the universal vector bundle along $\theta(n)$. 
\begin{defn}
A \emph{$\theta$-structure} on a $d$-manifold $M$ is a bundle map (fibrewise isomorphisms) $\ell: TM \to V^\theta_{d}$. 
We write $\Bun (TM,\theta)$ for the space of bundle maps  $TM \to V^\theta_d$, with the compact-open topology and define 
\begin{equation}\label{eqn:structuredBDiff}
B \Diff^{\theta}(M) := E \Diff(M) \times_{\Diff(M)} \Bun (TM,\theta) \to B \Diff(M).
\end{equation}
\end{defn}
This classifies smooth $M$-bundles $\pi: E \to X$, together with a bundle map $T_v E \to V_d^\theta$ from its vertical tangent bundle; the map \eqref{eqn:structuredBDiff} forgets the bundle data. We mention that $B \Diff^\theta(M)$ is often empty or disconnected and hence not the classifying space of a topological group; in \cite[Examples 1.32]{BotvinnikEbert} we worked out the homotopy type in some simple cases.

As explained in \cite[\S 1.7]{BotvinnikEbert}, it is conceptually useful to think about $B \Diff^\theta (M)$ as the space of all $(N,\ell)$, where $N$ is a closed manifold which is diffeomorphic to $M$ and $\ell: TN \to V_d^\theta$ is a $\theta$-structure on $N$. In light of this, a point $\ell \in \Bun (TM,\theta)$ gives rise to a basepoint $(M,\ell) \in B \Diff^{\theta}(M)$, and we define
\[
\Gamma^\theta (M,\ell) := \pi_1 (B \Diff^{\theta}(M),(M,\ell)).
\]
If the precise choice of $\ell$ does not matter, we just write $\Gamma^\theta(M)$. The map \eqref{eqn:structuredBDiff} induces on fundamental groups a homomorphism 
\begin{equation}\label{eqn:forgetfulhom}
\Gamma^\theta (M,\ell) \to \Gamma(M)
\end{equation}
which in general is neither injective nor surjective. 
\begin{rem}\label{rem:image-forgetmcg}
Three cases deserve special attention.
\begin{enumerate}
\item If $\theta:BO \to BO$ is the identity, $\Bun(TM,\theta) \simeq *$, so \eqref{eqn:forgetfulhom} is an isomorphism. 
\item If $\theta:B\SO \to BO$, $\ell$ corresponds to an orientation on $M$, and \eqref{eqn:forgetfulhom} is injective with image $\Gamma^+ (M):= \pi_0 (\Diff^+ (M))$. 
\item If $\theta: B \Spin \to B O$, $\ell$ corresponds to a spin structure on $M$. Then \eqref{eqn:forgetfulhom} has kernel of order $2$ and its image consists of those mapping classes which fix $\ell$ up to isomorphism. If in addition $M$ is $1$-connected, the image of \eqref{eqn:forgetfulhom} agrees with $\Gamma^+(M)$. 
\end{enumerate}
\end{rem}
Now assume that $f\in \Diff(M)$ is such that $[f] \in \Gamma(M)$ lies in the image of \eqref{eqn:forgetfulhom}. A choice of a preimage in $\Gamma^\theta(M,\ell)$ gives rise to a bundle map $T_v T(f) \to V_d^\theta$ from the vertical tangent bundle of the mapping torus. Since $T T(f) = T_v T(f) \oplus \bR$, we get a bundle map $TT(f) \to V_{d+1}^\theta$, which gives the manifold $T(f)$ a $\theta$-structure. We therefore get an element $[T(f)] \in \Omega_{d+1}^\theta$, the bordism group of $(d+1)$-manifolds with stable tangential $\theta$-structure\footnote{Examples (for $B=B\Spin$) show that $[T(f)]$ actually depends on the choice of the lift of $[f]$ along \eqref{eqn:forgetfulhom}. This won't cause any trouble for us.}. This construction gives a map 
\begin{equation}\label{eqn:mappingtorusmap-theta}
\Gamma^\theta (M,\ell) \to \Omega_{d+1}^\theta,
\end{equation}
generalizing \eqref{eqn:mappingtorusmap-SO}. Just as the map \eqref{eqn:mappingtorusmap-SO}, the more general map \eqref{eqn:mappingtorusmap-theta} is a group homomorphism (see \cite[Corollary 3.30]{Frenck} for the proof in the general case). 

\begin{thm}[Frenck \cite{Frenck}]\label{thm:georg}
Let $M$ be a closed manifold of dimension $d \geq 6$ and let $\ell \in \Bun(TM,\theta)$ be a $2$-connected $\theta$-structure on $M$\footnote{In other words, the map $M \to B$ underlying $\ell$ is $2$-connected}. Then there exists a group homomorphism $F: \Omega^\theta_{d+1} \to \pi_0 (\hAut(\Riem^+ (M)))$ such that the diagram 
\[
\xymatrix{
\Gamma^\theta(M,\ell) \ar[d]_{\eqref{eqn:forgetfulhom}} \ar[r]^-{\eqref{eqn:mappingtorusmap-theta}} & \Omega^\theta_{d+1} \ar[d]^{F}\\
\Gamma(M) \ar[r]^-{\eqref{eqn:actionhomomorphism}} & \pi_0 (\hAut(\Riem^+ (M)))
}
\]
commutes. 
\end{thm}

The main ingredients for Theorem \ref{thm:georg} are the Gromov--Lawson--Chernysh surgery theorem \cite[Theorem 1.2]{EbertFrenck} and Igusa's $2$-index theorem \cite[Chapter VI]{Igusa}. 

Theorem \ref{thm:georg} was preceeded by \cite[Theorem 4.1.2]{BERW} and \cite[Theorem 4.1.2]{ERW19}. Both these results established (for some special $M$) diagrams similar to the one in \ref{thm:georg}, but with an a priori unknown abelian group in place of $\Omega^\theta_{d+1}$. 

\begin{rem}
Let $[f] \in \Gamma^\theta(M,\ell)$. 
We emphasize that the theorem gives a \emph{sufficient} condition for $f^*\sim \id \in \hAut(\Riem^+(M))$: if the map underlying $\ell$ is $2$-connected and if $[f]$ maps to $0\in \Omega_{d+1}^\theta$ under \eqref{eqn:mappingtorusmap-theta}, then $f^* \simeq \id \in \hAut (\Riem^+(M))$. 

A \emph{necessary} condition for $f^* \simeq \id$ is as follows. Regardless of $d \geq 6$ and $\ell$ being $2$-connected or not, $f^* \sim \id$ implies that the mapping torus $T(f)$ carries a psc metric. It follows that the image of $[f]$ under \eqref{eqn:mappingtorusmap-theta} lies in the subgroup $(\Omega_{d+1}^\theta)_+ \subset \Omega_{d+1}^\theta$ of cobordism classes having a psc representative. 

Therefore for $[f] \in \Gamma^\theta(M,\ell)$, we deduce
\[
[T(f)] = 0 \in \Omega_{d+1}^\theta \stackrel{d \geq 6, \ell \, 2-\text{connected}}{\Rightarrow} f^* \sim \id \Rightarrow [T(f)] \in (\Omega_{d+1}^\theta)_+. 
\]
In the case of $\theta: B \Spin \to BO$, the group $(\Omega_{d+1}^\Spin)_+ \subset \Omega_{d+1}^\Spin$ coincides with the kernel of the $\alpha$-invariant $\ahat: \Omega_{d+1}^\Spin \to \KO_{d+1}$, by the main result of \cite{Stolz}. Usually $(\Omega_{d+1}^\Spin)_+$ is not the trivial subgroup. 

More generally, one can phrase the necessary condition more succinctly in terms of the Stolz exact sequence from \cite{Stolz2}
\[
\ldots \to \Omega_{d+1}^{\theta,\psc} \stackrel{F}{\to} \Omega_{d+1}^{\theta} \stackrel{G}{\to} R_{d+1}^\theta \stackrel{\partial}{\to} \Omega_{d}^{\theta,\psc} \to \ldots.
\]
Here $\Omega_{d}^{\theta,\psc}$ is the cobordism group of closed $d$-dimensional $\theta$-manifolds with psc metrics (and psc cobordisms); $F$ forgets the psc metrics and $(\Omega_{d+1}^\theta)_+= \im (F)$; $R_{d+1}^\theta$ is a relative group of compact $(d+1)$-dimensional $\theta$-manifolds with psc metrics on the boundary and $\partial$ takes the boundary. So for $[f] \in \Gamma^\theta(M,\ell)$, we deduce
\[
[T(f)] = 0 \in \Omega_{d+1}^\theta \stackrel{d \geq 6, \ell \, 2-\text{connected}}{\Rightarrow} f^* \sim \id \Rightarrow G([T(f)])=0\in R_{d+1}^\theta. 
\]
One might speculate whether one of the above implications is an equivalence, but that is a wide open question.
\end{rem}

\subsection{Bantje's theorem}

The case $n=1$ of Theorem \ref{thm:rigidity-dimension4} requires us to use the more involved (to prove, state and use) rigidity result \cite[Theorem 10.4.1]{Bantje}, which is stronger that Theorem \ref{thm:georg} in three aspects.
\begin{enumerate}
\item It applies in dimension $d\geq 5$ as opposed to $d \geq 6$,
\item it provides a space--level variant of the diagram in Theorem \ref{thm:georg},
\item it applies to manifolds with nonempty boundary (we won't discuss this here). 
\end{enumerate}
However, in the case where both results apply, Theorem \ref{thm:georg} has a slightly stronger conclusion (reflected in the stronger hypotheses of Theorem \ref{thm:rigidity-dimension4} in the case $n=1$). 
Having given this brief overview, let us turn to the actual statement, quoting \cite[Theorem 10.4.1]{Bantje} (and explaining its constitutents afterwards). 

\begin{thm}[Bantje]\label{thm:jannes}
Let $M$ be a closed manifold of dimension $d \geq 5$ and let $\theta: B \to BO$ be a map. Then there is a homotopy cartesian diagram
\begin{equation}\label{eqn:jannes}
\xymatrix{
\Riem^+ (M) \hq \Diff(M) \ar[d]  & B \Diff^{\theta,\psc} (M)^{(2)} \ar[d] \ar[r] \ar[l] & \fX \ar[d] \\
B \Diff (M) & B \Diff^\theta (M)^{(2)} \ar[l] \ar[r]^{\alpha} & \Omega^\infty \MT \theta(d). 
}
\end{equation}
\end{thm}

\begin{explanations}
\begin{enumerate}
\item The left column of \eqref{eqn:jannes} is just the Borel construction 
\[
\Riem^+ (M) \hq \Diff(M)  := E \Diff(M) \times_{\Diff(M)} \Riem^+ (M)\to B \Diff(M)
\]
coming from the diffeomorphism action. 
\item $B \Diff^\theta (M)^{(2)} \subset B \Diff^\theta (M)$ is a union of path components. More precisely, write $\Bun (TM,\theta)^{(2)} \subset \Bun (TM,\theta)$ for those bundle maps whose underlying map $M \to B(d)$ is $2$-connected, and put 
\[
B \Diff^\theta (M)^{(2)}:=E \Diff(M) \times_{\Diff(M)} \Bun (TM,\theta)^{(2)}. 
\]
The lower left horizontal map of \eqref{eqn:jannes} is the restriction of \eqref{eqn:structuredBDiff}, and the left square of \eqref{eqn:jannes} is defined to be a pullback. In other words
\[
B \Diff^{\theta,\psc} (M)^{(2)} := E \Diff(M) \times_{\Diff(M)} (\Bun (TM,\theta)^{(2)} \times \Riem^+ (M)). 
\]
\item $\MT \theta(d)$ is the Madsen--Tillmann spectrum from \cite{GMTW}, in other words the Thom spectrum of the additive inverse of $V_d^\theta\to B(d)$. For any $d$-manifold $M$, we have the \emph{Madsen--Tillmann map}
\[
\alpha: B \Diff^\theta (M) \to \Omega^\infty \MT \theta(d)
\]
also (in this generality only implicitly) discussed in \cite{GMTW}. The right horizontal bottom map of \eqref{eqn:jannes} is the restriction of $\alpha$ to $B \Diff^\theta (M)^{(2)}$. 
\item The space $\fX$ is constructed in \cite{Bantje} and (of course) the main point of the theorem, and also the place where the $2$-connectivity condition enters. In \S \ref{rem:outlinejannes} below, we give a brief sketch. 
\item If either $\Riem^+ (M) =\emptyset$ or $\Bun(TM,\theta)^{(2)} = \emptyset$, the result is vacuously true (with $\fX=\emptyset$). Hence a sensible choice of $\theta$ is crucial when applying the result.
\end{enumerate}
\end{explanations}

\begin{rems}
\begin{enumerate}
\item The unstated version of Theorem \ref{thm:jannes} for manifolds with boundary supersedes \cite[Theorems E and F]{ERW22} which apply only for $d \geq 6$ and \emph{exclusively} to manifolds with nonempty boundary. The latter result implies the main result of \cite{BHSW} which might be viewed as the first instance of result in this direction. With hindsight, these theorems have been the main obstacle in the attempt to directly use the topology of $\Diff (M)$ to get information about $\Riem^+(M)$ \cite{Hitchin}, \cite{CrowSchick}, \cite{HSS}.
\item In some special cases, diagrams similar to \eqref{eqn:jannes} have been established in \cite[Diagram (1.6)]{BERW} and \cite[Theorem F]{ERW19}, as consequences of the special cases of Theorem \ref{thm:georg} proved in those papers. These have been the decisive steps in those papers.
\end{enumerate}
\end{rems}

Let us next explain how to deduce from Theorem \ref{thm:jannes} a statement that is formally similar to that of Theorem \ref{thm:georg}. A $2$-connected $\theta$-structure $\ell\in \Bun(TM,\theta)^{(2)}$ on $M$ determines $(M,\ell) \in B \Diff^\theta(M)^{(2)}$; let $B \Diff^\theta(M)^{(2)}_{(M,\ell)}$ be the path component containing that basepoint. We restrict \eqref{eqn:jannes} to obtain the diagram
\begin{equation}\label{eqn:jannes1}
\xymatrix{
\Riem^+ (M) \hq \Diff(M) \ar[d]  & B \Diff^{\theta,\psc} (M)^{(2)}_{(M,\ell)} \ar[d] \ar[r] \ar[l] & \fX_{(M,\ell)} \ar[d] \\
B \Diff (M) & B \Diff^\theta (M)^{(2)}_{(M,\ell)} \ar[l] \ar[r]^{\alpha} & \Omega^\infty_{(M,\ell)} \MT \theta(d);
}
\end{equation}
here $\Omega^\infty_{(M,\ell)} \MT \theta(d) \subset \Omega^\infty \MT \theta(d)$ is the path component containing the image of $(M,\ell)$ under $\alpha$. The upper row is obtained by taking the preimages of the distinguished path components; hence the diagram \eqref{eqn:jannes1} is still homotopy cartesian. 

The left column of \eqref{eqn:jannes1} is by definition a fibration with fibre $\Riem^+ (M)$. As such, it has a classifying map $B \Diff(M) \to B \hAut(\Riem^+ (M))$, unique up to homotopy, by the general classification theory for fibrations \cite{May}. Of course, this map just comes from delooping the action map $\Diff(M) \to \hAut(\Riem^+ (M))$ and does not carry any new information. The composition $B \Diff^\theta(M)^{(2)}_{(M,\ell)} \to B \Diff(M)\to B \hAut(\Riem^+ (M))$ is then, tautologically, a classifying map for the middle fibration in \eqref{eqn:jannes1}. The fact that the right square in \eqref{eqn:jannes1} is homotopy cartesian implies that the classfying map for the middle fibration factors, up to homotopy, through $\Omega^\infty_{(M,\ell)}\MT \theta(d)$. Therefore we establish from \eqref{eqn:jannes1} a homotopy commutative diagram 
\begin{equation}\label{thm:jannesgeorg}
\xymatrix{
B \Diff^\theta (M)^{(2)}_{(M,\ell)} \ar[d] \ar[r]^{\alpha} & \Omega^\infty_{(M,\ell)} \MT \theta(d) \ar[d]^{\varphi}\\ 
B \Diff(M) \ar[r] & B \hAut(\Riem^+ (M)).
}
\end{equation}
Taking fundamental groups at $(M,\ell)\in B \Diff^\theta(M)^{(2)}$ and its image points under the maps in \eqref{thm:jannesgeorg} gives a commutative diagram\footnote{The Madsen--Tillmann map $\alpha:B \Diff^\theta(M) \to \Omega^\infty \MT \theta(d)$ sends $(M,\ell)$ to a $\alpha(M,\ell) \in \Omega^\infty \MT \theta(d)$, which is the basepoint for the fundamental group in the upper right corner.} 
\begin{equation}\label{thm:jannesgeorggroups}
\xymatrix{
\Gamma^\theta (M,\ell) \ar[d] \ar[r]^-{\alpha_*} & \pi_1 (\Omega^\infty \MT \theta(d),\alpha(M,\ell)) \ar[d]^{\varphi}\\ 
\Gamma (M)\ar[r] & \pi_0 (\hAut(\Riem^+ (M)))
}
\end{equation}
of groups. Since $\Omega^\infty \MT \theta(d)$ is an infinite loop space, the fundamental group at $\alpha(M,\ell)$ is canonically identified with the fundamental group at the basepoint which in turn is identified with the homotopy group $\pi_1 (\MT \theta(d))$ of the \emph{spectrum} $\MT \theta(d)$ (keeping track of this identification will require some care further below).

\subsection{A very brief overview of the proof of Bantje's theorem}\label{rem:outlinejannes}

Let us mention that the proof of Theorem \ref{thm:jannes} relies on ideas of Perlmutter \cite{Perlmutter1}, \cite{Perlmutter2} which in turn are an elaboration of large parts of the original proof of the Mumford conjecture by Madsen and Weiss \cite{MadsenWeiss} in high dimensions. Given the well-known complexity of \cite{MadsenWeiss}, the reader should not expect an easy ride. 

One first introduces a certain spaces of manifolds $\cW_{\theta,T}^{[3]}$ and $\cW_{\theta,T}^{[3],\psc}$. Here $T$ is a finite set, equipped with some extra data. 
The points in $\cW_{\theta,T}^{[3]}$ are $d$-dimensional closed manifolds with $\theta$-structures and a collection of surgery data of index in $\{3,\ldots,d-2\}$, indexed by the elements of $T$. The $\theta$-structures are required to be $2=3-1$-connected. The points of $\cW_{\theta,T}^{[3],\psc}$ are the points of $\cW_{\theta,T}^{[3]}$, but with psc metrics as additional datum; the psc metrics need to be of standard form on the surgery data. 

Almost by definition, the forgetful map $\cW_{\theta,\emptyset}^{[3],\psc} \to \cW_{\theta,\emptyset}^{[3]}$ is the disjoint union of the middle maps in \eqref{eqn:jannes}, over all closed $M$. 

As $T$ ranges through some indexing category $\cK^{3}$ of finite sets with extra data, the space $\cW_{\theta,T}^{[3]}$ and its psc variant depend functorially on $T$ (by either performing surgeries along the surgery data or by dropping the surgery data), and we can form the homotopy colimit over $\cK^{3}$. Together with the natural inclusion maps to the homotopy colimit, one obtains a commutative diagram
\begin{equation}\label{eqn:jannesproofdiagram}
\xymatrix{
B \Diff^{\theta,\psc} (M)^{(2)} \ar[r]\ar[d] & \cW_{\theta,\emptyset}^{[3],\psc} \ar[r] \ar[d] & \hocolim_{T \in \cK^3} \cW_{\theta,T}^{[3],\psc} \ar[d] \\
B \Diff^{\theta} (M)^{(2)} \ar[r]& \cW_{\theta,\emptyset}^{[3]} \ar[r] & \hocolim_{T \in \cK^3} \cW_{\theta,T}^{[3]} . \\
}
\end{equation}
The left square is homotopy cartesian almost by definition. The right square is homotopy cartesian as well: this uses a well-known general criterion for proving cartesianness of such squares; the hypothesis of that criterion holds in the case at hand by the Gromov--Lawson--Chernysh theorem. 

The crucial point in the proof is to show that the bottom composition in \eqref{eqn:jannesproofdiagram} can be factored as 
\begin{equation}\label{eqn:composition-jannes-proof}
B \Diff^{\theta} (M)^{(2)} \to \Omega^\infty \MT \theta(d) \to  \hocolim_{T \in \cK^3} \cW_{\theta,T}^{[3]}.
\end{equation}
Once this is done, we define $\fX$ as the homotopy pullback of the right column of \eqref{eqn:jannesproofdiagram} along the second map of \eqref{eqn:composition-jannes-proof}, and Theorem \ref{thm:jannes} follows. The construction of the factorization \eqref{eqn:composition-jannes-proof} involves an elaborate parametrized surgery argument, generalizing the more mysterious portions of \cite{MadsenWeiss} to high dimensions, and we won't attempt an informative synopsis here.

\subsection{First part of the proof of Theorem \ref{thm:rigidity-dimension4}}

\begin{prop}\label{prop:proofmainthm:generaloutline}
Let $M^d$ be a closed manifold with a $2$-connected $\theta$-structure $\ell:M \to B$, and let $[f] \in \Gamma^\theta(M,\ell)$ be given. Let $N^n$ be a further closed manifold. Then the map $(f \times \id_N)^*: \Riem^+ (M \times N) \to \Riem^* (M \times N)$ is homotopic to the identity, under each of the following hypotheses.
\begin{enumerate}
\item $d+n \geq 6$ and the cobordism class of the mapping torus $[T(f)] \in \Omega_{d+1}^\theta$ vanishes.
\item $d+n \geq 5$, and the element $\alpha_*(f) \in \pi_1 (\Omega^\infty \MT \theta(d),\alpha(M,\ell))$ vanishes.
\end{enumerate}
\end{prop}

\begin{proof}
(1) Given a further map $\vartheta:C \to BO$, we let 
\[
\theta \oplus \vartheta : B \times C \stackrel{\theta \times \vartheta}{\to} BO \times BO \stackrel{\mu}{\to} BO
\]
where $\mu$ is the direct sum map. There is a bilinear map 
\[
\_ \times \_ : \Omega_{d}^{\theta} \times \Omega_{n}^{\vartheta} \to \Omega_{d+n}^{\theta\oplus  \vartheta}
\]
taking products of representatives. If $N^n$ comes equipped with a $\vartheta$-structure, then $[f \times \id_N] \in \Gamma (M \times N)$ can be lifted to $\Gamma^{\theta\oplus \vartheta} (M \times N)$, and 
\[
[T(f \times \id_N)] = [T(f)] \times [N]  \in \Omega_{d+n}^{\theta \oplus \vartheta}.
\]
Since $[T(f)]=0$, we get $[T(f \times \id_N)]=0$. 

Theorem \ref{thm:georg} gives us the desired conclusion, provided that we can choose $C$ so that the product tangential structure $M \times N \to B \times C$ is $2$-connected. The original structure $\ell: M \to B$ is $2$-connected by assumption, and one possibility for $C$ is just the Gauss map $\vartheta:N \to BO$ of $TN$. Then $m= \id: N \to C=N$ is a $\vartheta$-structure which is clearly $2$-connected.

(2) is almost identical to (1). We let $N \stackrel{m}{\to} C \stackrel{\vartheta}{\to} BO$ as just constructed. There is an external product structure 
\begin{equation}\label{pairing-MTspectra}
\MT \theta(d) \wedge \MT \vartheta (n) \to \MT (\theta\oplus\vartheta) (d+n)
\end{equation}
that we use as follows. The right square of the diagram 
\begin{equation}\label{eqn:priductdiagramMTmaps}
\xymatrix{
S^1 \ar[r]^-{(f,N)} \ar@{=}[d] & B \Diff^\theta(M) \times B \Diff^\vartheta (N) \ar[d] \ar[r]^{\alpha_M \times \alpha_N} & \Omega^\infty \MT \theta (d) \times \Omega^\infty MT \vartheta (n) \ar[d]\\
S^1  \ar[r]^-{f\times \id_N} & B \Diff^{\theta \oplus \vartheta} (M \times N) \ar[r]^{\alpha_{M \times N}} & \Omega^\infty \MT (\theta\oplus\vartheta) (d+n)
}
\end{equation}
commutes up to homotopy; the middle vertical map comes from taking direct products of bundles, the right vertical map is induced from \eqref{pairing-MTspectra} and the two horizontal maps are the Madsen--Tillmann maps.
The two left horizontal maps denoted $f$ and $f \times \id_N$ are the classifying maps for the mapping tori of $f$ and $f \times \id_N$, and $N: S^1 \to B \Diff^\vartheta (N)$ is the constant map to the basepoint. The hypothesis of the proposition implies that the top composition in \eqref{eqn:priductdiagramMTmaps} is homotopic to a constant map; hence so is the bottom composition. In light of \eqref{thm:jannesgeorggroups} (which is applicable as $\dim (M \times N) \geq 5$), this concludes the proof.
\end{proof}

\begin{proof}[Proof of Theorem \ref{thm:rigidity-dimension4} for $n\geq 2$]
Depending on whether $M$ admits a spin structure or not, let $\theta= \Spin$ or $\theta=\SO$. In both cases, we obtain $\theta$-structures $\ell: M \to B$ which are $2$-connected. Also in both cases, the class of $[f] \in \Gamma(M)$ admits a lift to $\Gamma^\theta(M)$ (use Remark \ref{rem:image-forgetmcg}), which we also denote $[f]$. In both cases, we have 
\[
[T(f)] =0 \in \Omega_5^\theta;
\]
(if $\theta=\SO$, this is assumed, if $\theta=\Spin$, just use $\Omega_5^\Spin=0$ \cite{MilnorSpin}).
Then apply Proposition \ref{prop:proofmainthm:generaloutline} (1).
\end{proof}

\section{The bordism computation}\label{sec:calculation}

To complete the proof of Theorem \ref{thm:rigidity-dimension4} in the case $n=1$ , it remains to show that the conditions of (3) and (4)  imply that $\alpha_*(f) =0 \in \pi_1 (\Omega^\infty \MT \theta(4),\alpha(M,\ell))$ in both cases. This requires a calculation of $\pi_1 (\MT \SO(4))$ and $\pi_1 (\MT \Spin (4))$.
Let 
\begin{equation}\label{eqn:dentildealpha}
\tilde{\alpha}: \Gamma^\theta (M^4,\ell) \stackrel{\alpha_*}{\to} \pi_1 (\Omega^\infty \MT \theta(4),\alpha(M,\ell)) \cong \pi_1 (\MT \theta (4))
\end{equation}
be the composition of the induced map $\alpha_*$ appearing in \eqref{thm:jannesgeorggroups} with the canonical isomorphism that comes from using the $h$-space structure and the identification of the fundamental group at the basepoint of an infinite loop space with the first homotopy group of the spectrum.

The way we stated Theorem \ref{thm:rigidity-dimension4} involves the Lie group $\Aut (I_M)$, so let us have a closer look at that. Let more generally $(V,b)$ be a finite-dimensional $\bR$-vector space with a nondegenerate symmetric bilinear form. 
By Sylvester's law of inertia, there is a basis of $V$ with respect to which $b$ is given by the matrix $\twomatrix{1_p}{}{}{-1_q}$; so $\Aut(b) \cong O(p,q)$. Any element $g \in O(p,q)$ can be written as 
\[
\twomatrix{g_+}{g_{01}}{g_{10}}{g_-}
\]
and the elements $g_+$, $g_-$ are necessarily invertible. Therefore, the functions $\delta_\pm: g \mapsto \sign (\det (g_\pm))$ are locally constant. We get continuous functions $\delta_\pm: \Aut(b) \to \{\pm 1\}= \bZ/2$, and by restricting to the maximal compact subgroup which is isomorphic to $O(p) \times O(q)$, one proves that these are group homomorphisms, and that $\delta_+(g) \delta_-(g)=\det (g)$. It is well-known that 
\[
(\det,\delta_+):\pi_0 (\Aut(b)) \to \bZ/2^2
\]
is injective, and surjective if $b$ is not definite. 

With these prelimaries in place, we can now state the result of this section, which by Proposition \ref{prop:proofmainthm:generaloutline} implies Theorem \ref{thm:rigidity-dimension4} in the case $n=1$.

\begin{prop}\label{prop:bordismcalcul}
There is an isomorphism 
\[
\varphi=(\varphi_1,\varphi_2,\varphi_3): \pi_1 (\MT \SO(4)) \to \bZ/2^3
\]
such that when $M^4$ is $1$-connected and $f \in \Diff^+ (M)$, we have
\begin{equation}\label{eqn:formula-invariant}
\varphi (\tilde{\alpha} (f)) = (w_2 w_3[T(f)],\det(H^2(f;\bR)),\delta_+ (H^2 (f;\bR))).
\end{equation}
The map $\pi_1 (\MT \Spin (4)) \to \pi_1(\MT \SO(4))$ is injective, with image $\ker (\varphi_1)$. 
\end{prop}

We remark that $\pi_4 (\MT \Spin (4)) \cong \bZ/2^2$ was shown in \cite{GRWAb} (combine the sequence (5.2), Lemma 5.2 and 5.5 of loc.cit. with the fact that $\Omega_5^\Spin =0$). Since that paper does not cover $\pi_1 (\MT \SO(4))$, we give an alternative proof from scratch.

The proof is a bit convoluted and we summarize the main steps beforehand.  
\begin{enumerate}
\item In \S \eqref{subsec:kertvaire}, we recall the Kervaire semicharacteristic and prove \eqref{eqnw2w3mappingtorus}.
\item In \S \eqref{subsection:estimate}, we use basic facts about low-dimensional homotopy groups of Madsen--Tillmann spectra. These could be used to prove that $\pi_1 (\MT \SO(4))\cong \bZ/2^3$ directly. For us it is however only necessary to estimate the order of the groups as in Lemma \ref{lem:grouporder}.
\item We conclude the argument in \S \ref{subsec:conclusionargument} by constructing $\varphi$ with the claimed properties. Explicit examples prove that $\varphi$ is surjective, and hence injective by (2). The claim about $\pi_1 (\MT \Spin (4))$ will then also follow immediately.
\end{enumerate}

\subsection{The Kervaire semicharacteristic}\label{subsec:kertvaire}

\begin{defn}
Let $M$ be a closed manifold of dimension\footnote{We are interested in the case $m=1$.} $4m+1$ and let $\bF$ be a field. The \emph{Kervaire semicharacteristic} $\Kerv{\bF}(M) \in \bZ/2$ of $M$ over $\bF$ is 
\[
\Kerv{\bF} (X):= \sum_{k=0}^{2m} \dim (H^k (M;\bF)) \pmod 2\in \bZ/2. 
\]
\end{defn}
Clearly $\Kerv{\bF}(M)$ only depends on the charateristic of $\bF$. 
It is a theorem of Lusztig, Milnor and Peterson \cite{LusMilPet} that 
\begin{equation}\label{eqn:milluspet}
\Kerv{\bR} (M)- \Kerv{\bF_2}(M) = \scpr{w_2 (TM) w_{4m-1} (TM),[M]}
\end{equation}
for closed oriented $M^{4m+1}$. Therefore, 
\begin{equation}\label{eqn:milluspet2}
\Kerv{\bR} - \Kerv{\bF_2}: \Omega^\SO_5 \to \bZ/2
\end{equation}
is well-defined (note that neither of the summands by itself is bordism invariant). Moreover, \eqref{eqn:milluspet2} is an isomorphism. The Kervaire semicharacteristic of a mapping torus $T(f)$ is easily computed in terms of the action of $f$ on cohomology, as follows. 

\begin{lem}\label{lem:kervaire-mappingtorus}
Let $f:M \to M$ be a diffeomorphism of a closed $4m$-manifold. Then 
\[
\Kerv{\bF}(T(f))= \dim (\Eig_{1} (H^{2m} (f;\bF))) \pmod 2\in  \bZ/2. 
\]
\end{lem}

Together with \eqref{eqn:milluspet}, Lemma \ref{lem:kervaire-mappingtorus} implies the formula
\eqref{eqnw2w3mappingtorus}. 

\begin{proof}
The Leray--Serre spectral sequence for $T(f) \to S^1$
\[
E^{p,q}_2 = H^p (S^1;H^q (M;\bF)) \Rightarrow H^{p+q}(T(f);\bF)
\]
collapses for degree reasons. Here $H^p (S^1;H^q (M;\bF))$ denotes the cohomology twisted by the action of $\pi_1 (S^1)=\bZ$ induced by $H^q (f;\bF)$. 
Therefore
\begin{equation}\label{eqn:kervaire-mappngtorus}
\Kerv{\bF} (T(f)) =  \dim (E^{0,0}_2) + \sum_{k=1}^{2m} (\dim (E^{0,k}_2) + \dim (E^{1,k-1}_2) )\pmod 2= 
\end{equation}
\[
= \sum_{k=0}^{2m-1} (\dim (E^{0,k}_2) + \dim (E^{1,k}_2)) + \dim (E^{0,2m}_2)\pmod 2\in \bZ/2.
\]
For each $q$, we have 
\[
\dim (E^{0,q}_2) - \dim (E^{1,q}_2) = \chi(S^1;H^q (M;\bF)) =0\in \bZ
\]
by the well-known formula $\chi(X;F)= \chi(X) \dim (F)$ for the Euler number of a finite complex $X$, twisted by a coefficient system of finite-dimensional vector spaces. Hence \eqref{eqn:kervaire-mappngtorus} becomes
\[
\Kerv{\bF} (T(f))= \dim (E^{0,2m}_2)\pmod 2\in \bZ/2;
\]
and 
\[
E^{0,2m}_2= H^0 (S^1;H^{2m}(M;\bF)) = H^{2m}(M;\bF)^\bZ = \Eig _1(H^{2m}(f;\bF)).
\]
\end{proof}

\begin{lem}\label{lem:gantmacher}
Let $M^{4m}$ be closed and oriented, and let $f \in \Diff^+ (M)$. Then 
\[
(-1)^{\Kerv{\bR} (T(f))} (-1)^{\Kerv{\bR} (S^1 \times M)} = \det (H^{2m} (f;\bR) \in \pm 1. 
\]
\end{lem}

\begin{proof}
This is linear algebra: write $V:= H^{2m}(M;\bR)$, $b:= I_M$ and $g:= H^{2m}(f;\bR) \in \Aut(V,b)$. Pass to the complexification $V_\bC$ of $V$ and pick an isomorphism $V_\bC \cong \bC^n$ which transforms $b_\bC$ to the standard bilinear form $(x,y) \mapsto x^\top y$; then $g_\bC$ is represented by a matrix $A \in O(n,\bC)$. Theorem 9 of \cite[\S XI.5]{Gantmacher} describes the possible Jordan normal forms of $A\in O(n,\bC)$, and from the description given there, the claimed formula follows by an easy case-by-case check. 
\end{proof}

\subsection{Low--dimensional homotopy groups of Madsen--Tillmann spectra}\label{subsection:estimate}

Since $\MT \theta(d)$ is a Thom spectrum, its homotopy groups can be described in terms of cobordism groups. The following is a straightforward consequence of the classical Pontrjagin--Thom theorem as described in full generality in e.g. \cite[Chapter IV \S 7]{Rudyak}, \cite[\S II]{Stong}. 

\begin{lem}\label{lem:homotopygroupMTSpectrum}
For $k \in \bZ$, the group $\pi_k (\MT \theta(d))$ is the cobordism group of triples $[M^{d+k},\ell,\varphi]$, where $M^{d+k}$ is a closed manifold, $\ell:M \to B(d)$ is a map and $\varphi: TM \to \ell^* V_d^\theta \oplus \bR^k$ is a \emph{stable} bundle isomorphism\footnote{More abstractly, $\ell$ and $\varphi$ together are a point in $\colim_{n\to \infty} \Bun (TM\oplus \bR^n,V_d^\theta \oplus \bR^{k+n})$.}.
\end{lem}

Information about the low--dimensional homotopy groups of Madsen--Tillmann spectra can be obtained from the cofibre sequence \cite{GMTW}
\begin{equation}\label{cofibre-madsentillmann}
\MT \theta(n+1) \to \Sigma^\infty B(n+1)_+ \to \MT \theta(n)
\end{equation}
of spectra, which is natural in $\theta$. The connecting map $\MT \theta(n) \to \Sigma \MT \theta(n+1)$ is $0$-connected, and the colimit $\colim_n \Sigma^n \MT \theta(n)$ can be identified with the usual Thom spectrum $M \theta$. It follows that for $k<0$
\[
\pi_k (\MT\theta(n)) \cong \pi_{n+k} (M \theta)\cong \Omega^\theta_{n+k}
\]
is the ordinary bordism group of $(n+k)$-manifolds with stable \emph{tangential} $\theta$-structure (classical bordism theory uses $\theta$-structures on the stable \emph{normal} bundle).

If $B$ is $0$-connected (as it is in the cases $B= B\SO$ and $B=\Spin$ we are interested in), we obtain from \eqref{cofibre-madsentillmann} an exact sequence 
\begin{equation}\label{eqn:pinullmtspectra}
\xymatrix{
\pi_0 (\MT \theta(n+1)) \ar[d] &   \\
\pi_0 (\Sigma^\infty B(n+1)_+) \ar[d] \ar[r]^-{\cong} & \bZ \\
\pi_0 (\MT \theta(n))  \ar[d] & \\
\pi_{-1} (\MT \theta(n+1)) \ar[d] \ar[r]^-{\cong} & \Omega_n^\theta \\
\pi_{-1} (\Sigma^\infty B(n+1)_+) \ar@{=}[r] & 0.
}
\end{equation}
Using the description of $\pi_k (\MT \theta (n))$ in terms of cobordism classes from Lemma \ref{lem:homotopygroupMTSpectrum} and using the identification $\pi_0 (\Sigma^\infty B_+)=\bZ$, we can describe the maps in \eqref{eqn:pinullmtspectra} as follows. 
 
\begin{itemize}
\item The map $\pi_0 (\MT\theta(n)) \to \pi_{-1} (\MT \theta(n+1)) = \Omega_n^\theta$ sends $(M,\ell,\varphi)$ to $(M,\iota_n \circ \ell, \id_{\bR} \oplus \varphi)$ where $\iota_n: B(n) \to B(n+1)$ is the inclusion. 
\item The map $\pi_0 (\MT \theta(n+1)) \to \pi_0 (\Sigma^\infty B(n+1)_+) \cong \bZ$ sends $(M,\ell,\varphi)$ to the number $\scpr{e (\ell^* V_{n+1}^{\theta}),[M]} \in \bZ$ ($e$ denotes the Euler class). 
\item $\bZ \cong \pi_0 (\Sigma^\infty B(n+1)_+) \to \pi_0 (\MT \theta(n))$ sends $1$ to the class of $S^n$, with the constant map $S^n \to B(n)$ and the usual isomorphism $TS^n \oplus \bR \cong \bR^{n+1}$.
\end{itemize}

\begin{lem}\label{lem:pinullmt(5)}
There is a commutative diagram with exact rows
\[
\xymatrix{
0 \ar[r] & \bZ/2 \ar@{=}[d]\ar[r] & \pi_0 (\MT \Spin(5)) \ar[d] \ar[r] & \Omega_5^\Spin\ar[d] \ar[r] & 0\\
0 \ar[r] & \bZ/2\ar[r] & \pi_0 (\MT \SO(5)) \ar[r] & \Omega^\SO_5 \ar[r] & 0 .
}
\]
For any field $\bF$, the map $[M,f,\varphi] \mapsto \Kerv{\bF}(M)$ gives a well-defined splitting of the maps $\bZ/2 \to \pi_0(\MT\Spin(5)), \, \pi_0 (\MT \SO(5))$.
Hence 
\[ 
\pi_0 (\MT \Spin (5)) \cong \bZ/2; \; \pi_0 (\MT \SO (5))
\cong \bZ/2^2, 
\] 
and the map $\pi_0 (\MT \Spin (5)) \to \pi_0 (\MT \SO (5))$ is injective.
\end{lem}

\begin{proof}
For $B = B \SO$ or $B \Spin$, consider the portion
\[
\pi_0 (\MT \theta(6)) \to \bZ \to \pi_0 (\MT\theta(5))
\]
of the exact sequence \eqref{eqn:pinullmtspectra}. The first map takes $[N,\ell,\varphi]$ to 
\[
\scpr{e(\ell^* V_6^\theta),[N]} \equiv \scpr{w_6(\ell^* V_6^\theta),[N]} = \scpr{w_6 (TN),[N]} \equiv \chi(N) \equiv 0 \pmod 2
\]
and so maps into $2\bZ$. On the other hand $[S^6,*,\varphi]$, where the stable isomorphism $\varphi: \bR^6 \cong TS^6$ comes from the usual $TS^6 \oplus \bR = \bR^7$, is taken to $2$. This gives the commutative diagram claimed in the lemma. 

It is moreover proven in \cite[Proposition A.6]{Eb13} that the map $M \mapsto \Kerv{\bF} (M)$ assigning to $M$ its Kervaire semicharacteristic gives a well-defined homomorphism $\pi_0 (\MT \SO (5)) \to \bZ/2$, for each field $\bF$; the same applies to $\pi_0 (\MT \Spin(5))$ of course.
\end{proof}

To extract information about $\pi_1 (\MT \theta(4))$ from Lemma \ref{lem:pinullmt(5)}, we use the cofibre sequences $\MT\theta(5) \to B(5) \to \MT \theta(4)$ which give a commutative diagram
\[
\xymatrix{
\pi_1 (\Sigma^\infty B \Spin(5)_+) \ar[r]^-{\cong} \ar[d] & \pi_1 (\Sigma^\infty B \SO(5)_+) = \bZ/2 \ar[d]\\
\pi_1 (\MT \Spin (4)) \ar[r] \ar[d] & \pi_1 (\MT \SO (4)) \ar[d]\\
\pi_0 (\MT \Spin (5)) \ar[r] \ar[d] & \pi_0 (\MT \SO (5))\ar[d] \\
\bZ   \ar[r]^{=} & \bZ}
\]
with exact columns. The topmost isomorphisms hold since $B \Spin (5)$ and $B\SO(5)$ are $1$-connected and because $\pi_1^\st =\bZ/2$. The two $\pi_0$-groups are finite by Lemma \ref{lem:pinullmt(5)}, hence we may replace the $\bZ$'s in the bottom by $0$. Using Lemma \ref{lem:pinullmt(5)} once more, we obtain the commutative diagram \begin{equation}\label{eqn:rawdiag-pi1mt}
\xymatrix{
\bZ/2 \ar@{=}[r] \ar[d] & \bZ/2 \ar[d]\\
\pi_1 (\MT \Spin (4)) \ar[r] \ar[d] & \pi_1 (\MT \SO (4)) \ar[d]\\
\bZ/2 \ar[d] \ar[r] & \bZ/2 \oplus \bZ/2 \ar[d] \\
0    & 0}
\end{equation}
with exact columns. Therefore:
\begin{lem}\label{lem:grouporder}
The groups $\pi_1 (\MT \SO (4))$ and $\pi_1 (\MT \Spin (4))$ are finite $2$-groups of order at most $8$ and $4$, respectively.
\end{lem}

\subsection{Construction of the invariants and their evaluation}\label{subsec:conclusionargument}

It will be important for us to describe the composition
\begin{equation}\label{eqn:fromMCGtopi1mt}
\tilde{\alpha}: \Gamma^\theta (M,\ell) \stackrel{\alpha}{\to} \pi_1 (\Omega^\infty \MT \theta(d),M) \cong \pi_1 (\Omega^\infty \MT \theta(d),*) \cong \pi_1 (\MT \theta(d))
\end{equation}
in terms of the bordism-theoretic description of $\pi_1 (\MT \theta(d))$. A map $S^1 \to B \Diff^\theta(M)$ corresponding to $f \in \Gamma(M,\ell)$ classifies a fibre bundle $\pi: E =T(f) \to S^1$, together with a bundle map $T_v E \to V_d^\theta$. Using the standard trivialization of $TS^1$ and the isomorphism $TE \cong T_v E \oplus \pi^* TS^1$, we obtain a bundle map $TE \to V_d^\theta \oplus \bR$; in other words a map $m: E \to B(d)$, together with an isomorphism $TE \cong m^* V_d^\theta \oplus \bR$. These data represent an element $\beta(E) \in \pi_1 (\MT \theta(d))$. Given this observation, it is suggestive to believe that \eqref{eqn:fromMCGtopi1mt} sends $f$ to the class $\beta(T(f))$ of its mapping torus, but this is \emph{not true}. 

The second map in \eqref{eqn:fromMCGtopi1mt} is given by the infinite loop space structure. So we must subtract from the map $f: S^1 \to B \Diff^\theta (M)$ which is pointed at $(M,\ell)$ the constant map at this same basepoint which is nothing else that the mapping torus of the identity on $M$; in other words the trivial bundle $M \times S^1\to S^1$. Therefore, the correct formula for \eqref{eqn:fromMCGtopi1mt} is
\begin{equation}\label{formula-foreqn:fromMCGtopi1mt}
\tilde{\alpha} (f) = \beta(T(f)) - \beta(T(\id_M)). 
\end{equation}
We now define 
\[
\varphi_1: \pi_1 (\MT \SO(4)) \to \bZ/2
\]
to be the composition 
\[
\pi_1 (\MT \SO(4)) \to \pi_0 (\MT \SO(5)) \to \Omega_2^\SO \cong \bZ/2 ;
\]
using that $T(\id_M )=S^1 \times M$ is nullbordant, we get 
\[
\varphi_1 (\alpha_* (f)) = w_2 w_3 ([T(f)]-[T(\id_M)]) = w_2 w_3 ([T(f)]) . 
\]

To define $\varphi_2$ and $\varphi_3$, we exploit the relation between Madsen--Tillmann spectra and family index theory that we worked out in \cite{Eb13}. 
For a smooth bundle $\pi:E \to B$ of oriented $d$-dimensional closed manifolds, we denote by $H^k (E \to B;\bR) \to B$ the real vector bundle whose fibre over $b$ is $H^k(M;\bR)$. This is flat. If $d=4$, the bundle $H^2 (E \to B;\bR)$ can be written as a direct sum 
\[
H^2 (E \to B;\bR) \cong H^{2,+} (E \to B;\bR) \oplus H^{2,-} (E \to B;\bR),
\]
so that the (fibrewise) intersection form is positive/negative definite on the summands. This splitting is only natural up to homotopy and the individual summands are not flat. Some linear combinations of these bundles arise as family indices of elliptic operators; for example $\sum_{k=0}^d (-1)^d H^k(E \to B;\bR)\in KO^0(B)$ is the family index of the Euler characteristic operator, and similarly $H^{2,+} (E \to B;\bR) - H^{2,-} (E \to B;\bR) \in KU^0(B)$ is (when $d=4$) the family index of the signature operator. 

\begin{prop}\label{prop:familyindextheoremexampel}
There are spectrum maps $\chi,\sigma: \MT \SO(4) \to \KO$, such that for each bundle $\pi:E \to B$ of closed oriented $4$-manifolds on finite CW base, the induced maps
\[
\Omega^\infty \chi, \Omega^\infty \sigma: B \stackrel{\alpha_E}{\to} \Omega^\infty \MT \SO(4){\to} \Omega^\infty \KO \simeq \bZ\times BO
\]
are classifying maps for the virtual vector bundles
\begin{equation}\label{eqn:indexbundle1}
H^0 (E \to B;\bR)-H^1 (E \to B;\bR) + H^2 (E \to B;\bR)-H^3 (E \to B;\bR)+H^4 (E \to B;\bR)
\end{equation}
and 
\begin{equation}\label{eqn:indexbundle2}
H^0 (E \to B;\bR) - H^1  (E \to B;\bR) + H^{2,+} (E \to B;\bR), 
\end{equation}
respectively. 
\end{prop}
The maps $\chi$ and $\sigma$ induce maps on homotopy groups $\pi_1 (\MT \SO(4)) \to \pi_1 (\KO)=\bZ/2$. 
It follows that $\varphi=(\varphi_1,\chi_*,\sigma_*): \pi_1 (\MT \SO(4)) \to \bZ/2^3$ satisfies \eqref{eqn:formula-invariant} and that $\im (\pi_1 (\MT \Spin (4)) \to \pi_1(\MT \SO(4))$ is contained in $\ker (\varphi_1)$. 

\begin{proof}
For a closed Riemannian $d$-manifold $M$, let $0 \to \cA^0 (M) \stackrel{d}{\to} \ldots \stackrel{d}{\to} \cA^d (M)  \to 0 $ be the de Rham complex (we use $\bR$-valued forms). This is a real elliptic complex and the operator $D= d+d^*: \cA^{\mathrm{ev}}(M) \to \cA^{\mathrm{odd}}(M)$ is elliptic. This construction carries over to the family setting, and \eqref{eqn:indexbundle1} is the index bundle of $D$. 

To realize \eqref{eqn:indexbundle2} as the index bundle of a family of real elliptic operators, we use the \emph{selfduality complex}, a variant of the de Rham complex which exists on oriented $4$-manifolds and plays an important role in gauge theory, see e.g. \cite[p. 360 and 446]{Scorpan}.  
For oriented $M^4$, the Hodge star $\star: \Lambda^2 TM \to \Lambda^2 TM$ is an involution, and we let as usual $\Lambda^{2,\pm} TM$ be its $\pm 1$-eigenbundles and let $p_\pm : \Lambda^2 TM \to \Lambda^{2,\pm} TM$ be the projection maps. We define $d^+:= \sqrt{2} p_+ d: \cA^1 (M) \to \cA^{2,+}(M)$ (the factor of $\sqrt{2}$ is here for cosmetic reasons and makes the $D^+$ below into a Dirac operator); the selfduality complex is 
\begin{equation}\label{eqn:ellipticcompex}
0 \to \cA^0 (M)\stackrel{d}{\to} \cA^1 (M) \stackrel{d^+}{\to} \cA^{2,+}(M) \to 0.
\end{equation}
This is an elliptic complex, so 
\[
D^+:=d+(d^+)^*: \cA^0 (M)\oplus \cA^{2,+}(M) \to \cA^1 (M)
\]
is an elliptic operator. Using the Hodge decomposition theorem, it is not hard to see that 
\[
\ker (D^+)=\Harm^0(M) \oplus \Harm^{2,+}(M); \; \coker (D^+)=\Harm^1 (M)
\]
where $\Harm^k(M)$ denotes the space of harmonic $k$-forms and $\Harm^{2,+}(M) = \Harm^2 (M) \cap \cA^{2,+}(M)$. It follows that 
\[
\ind(D^+) = b_0(M) - b_1 (M) + b_2^+(M) \in \bZ. 
\]
Again, the construction carries over to families. If $E \to B$ is a bundle of $4$-dimensional oriented manifolds, we get a family of (real) elliptic operators $D^+$, and the family index is 
\[
\ind (D^+) = H^0 (E \to B;\bR) - H^1 (E \to B;\bR) + H^{2,+}(E \to B;\bR) \in KO^0 (B). 
\]
Using that $KO^0(B)=[B;\Omega^\infty\KO]$, we can view $\ind(D)$ and $\ind(D^+)$ as a homotopy class of maps $B \to \Omega^\infty \KO$. Expressing these homotopy classes in topological terms is the aim of the Atiyah--Singer family index theorem \cite{AtiyahSingerIV}, \cite{AtiyahSingerV}. In \cite{Eb13}, it was observed that the family index theorem can be reformulated in terms of Madsen--Tillmann spectra. 

The operators $D$ and $D^+$ are ``universal'' in the sense of \cite[\S 3.2]{Eb13}, there, we constructed universal symbol classes $\smb_D, \smb_{D^+} \in KO^0 (\MT\SO(4))$ (denoted $\mathrm{th}\sigma_D$ in loc. cit.). 

The classes $\smb_D$, $\smb_{D^+}$ are elements of the $K$-theory of the \emph{spectrum} $\MT \SO(4)$ and can be viewed as maps $ \MT \SO(4) \to \KO$ of spectra. Applying $\Omega^\infty$ produces maps $\Omega^\infty \MT \SO(4)\to \Omega^\infty \KO$. As shown in \cite[\S 3.2]{Eb13}, the Atiyah--Singer family index theorem says that $(\Omega^\infty \smb_D) \circ \alpha_E \sim \ind(D): B \to \Omega^\infty \KO$; and the analogous formula holds for $D^+$. 
We now put $\chi: =\smb_D$ and $\sigma:= \smb_{D^+}$. 

The above outline was slightly inaccurate because \cite{Eb13} only covered the case of \emph{complex} $K$-theory whereas here we need to use indices in real $K$-theory. One can modify the argument of that paper in order to cover the real case instead, using the real family index theorem \cite{AtiyahSingerV} in place of \cite{AtiyahSingerIV} and Real $K$-theory \cite{AtiyahKR}. 
Alternatively, the proof of the index theorem given in the more recent paper \cite{Eb19} deals with the real and complex case on the same footing. 
\end{proof}

Together with Lemma \ref{lem:grouporder}, the proof of Proposition \ref{prop:bordismcalcul} is completed as follows. 

\begin{lem}\label{lem:realizationofinvariants}
There exist orientation-preserving diffeomorphisms $f_j:M_j \to M_j$ of $1$-connected $4$-manifolds, $j=1,2,3$, such that $\{\varphi(\tilde{\alpha}(f_1)),\varphi(\tilde{\alpha}(f_2)),\varphi(\tilde{\alpha}(f_3))\}$ generates $\bZ/2^3$. The manifolds $M_1$ and $M_2$ are spin. 
\end{lem}

\begin{proof}
Let $M_1=M_2 = S^2 \times S^2$ which is spin, and let $f_1\in \Diff^+ (S^2 \times S^2)$ be the map that flips the two factors. With respect to a suitable basis of $H^2 (S^2 \times S^2;\bR)$, we have 
\[
I_{S^2\times S^2} = \twomatrix{0}{1}{1}{0}, \; H^2(f_1;\bR)= \twomatrix{0}{1}{1}{0}. 
\]
We have $\varphi_1 (\tilde{\alpha}(f_1))=0$ since $S^2 \times S^2$ is spin, and compute easily that 
\[
\varphi(\tilde{\alpha}(f_1)) = (0,1,0) \in \bZ/2^3.
\]
Let $f_2\in \Diff(S^2 \times S^2)$ be the product of two degree $-1$ diffeomorphisms. With respect to the above basis, we have 
\[
H^2(f_1;\bR)= \twomatrix{-1}{0}{0}{-1}
\]
and therefore 
\[
\varphi(\tilde{\alpha}(f_2)) = (0,0,1) \in \bZ/2^3.
\]
Let $M_3=\cp^2$ and let $f_3\in \Diff^+( \cp^2)$ be the complex conjugation. Then $I_{\cp^2}$ is positive definite, and
\[
\varphi(\tilde{\alpha}(f_3)) = (1,1,1) \in \bZ/2^3. 
\]
\end{proof}

\section{Proof of Theorem \ref{thm:hogherhomotopy}}

This is conceptually very similar to Theorem \ref{thm:rigidity-dimension4}. Let 
\[
x \in \ker (\pi_j (\Diff^+ (M)) \to \pi_j (\Homeo^+(M)))
\]
be as in Theorem \ref{thm:auckrub}. As in the proof of Proposition \ref{prop:proofmainthm:generaloutline} (look at the diagram \eqref{eqn:priductdiagramMTmaps} in particular), we must prove that the map 
\begin{equation}\label{aucklyrubmap}
\alpha_*:\pi_{j} (\Diff^+ (M)) = \pi_{j+1} (B \Diff^+ (M)) \to \pi_{j+1} (\Omega^\infty \MT \theta(4);\alpha(M)) \otimes \bQ
\end{equation}
annihilates $x$, where again $\theta= \SO$ or $\Spin$, depending on the case at hand. Let us first list recall the computation of rational homotopy and cohomology of $\Omega^\infty \MT \SO(4)$ and $\Omega^\infty \MT \Spin(4)$; see \cite[\S 2.4]{Eb13} for an exposition of this well-known material. We shall need the following facts.
\begin{enumerate}
\item The Thom isomorphism provides an isomorphism 
\[
\thom: H^*(B \SO(4);\bQ) \cong H^{*-4} (\MT\SO(4);\bQ)
\]
of the cohomology of $B \SO(4)$ with the cohomology of the \emph{spectrum} $\MT\SO(4)$ (this includes the case $*<4$). The cohomology suspension gives a map 
\[
\susp: H^* (\MT\SO(4);\bQ) \to H^* (\Omega^\infty\MT\SO(4);\bQ)
\]
to the cohomology of its infinite loop space. The cohomology of the infinite loop space is a graded commutative algebra, but the spectrum cohomology is only a graded vector space. The induced map of graded algebras
\[
\Lambda (\susp \circ \thom): \Lambda (H^{*>4} (B\SO(4);\bQ) \to H^* (\Omega^\infty_0 \MT \SO(4);\bQ)
\]
from the free graded commutative algebra on the source is an isomorphism.
\item For any class $c \in H^* (B \SO(4);\bQ)=\bQ[e,p_1]$ of degree $k>4$, we obtain a class 
\[
\lambda_c:= \susp (\thom(c)) \in H^{k-4}(\Omega^\infty\MT \SO(4);\bQ). 
\]
By step (1), $H^*(\Omega^\infty_0 \MT\SO(4);\bQ)$ is a polynomial algebra on all $\lambda_c$, where $c$ runs through all monomials on $e$ and $p_1$ of total degree $>4$. 
\item If $\pi:E \to B$ is a bundle of closed oriented $4$-manifolds with Madsen--Tillmann map $\alpha_E: B \to \Omega^\infty \MT \SO(4)$, we have 
\[
\alpha_E^* \lambda_c = \kappa_c (E) := \pi_!(c(T_v E)),
\]
the Miller-Morita--Mumford class of $E$ associated to $c$, obtained by fibre integration of the characteristic class $c$ evaluated on the vertical tangent bundle. 
\item Being a connected loop space, $\Omega^\infty_0 \MT \SO(4)$ has the property that the rational Hurewicz map $\pi_* (\Omega^\infty_0 \MT \SO(4)) \otimes \bQ \to H_* (\Omega^\infty_0 \MT \SO(4);\bQ)$ is injective. Hence a homotopy class $y \in \pi_j(\Omega^\infty_0 \MT \SO(4))$ vanishes after rationalization provided that $\scpr{\hur(y),x}=0$ for all $x \in H^* (\Omega^\infty_0 \MT \SO(4);\bQ)$. 
\item The same recipe applies to $\MT \Spin(4)$. Since the map $B \Spin(4) \to B \SO(4)$ induces an isomorphism in rational cohomology, the induced map 
\[
\Omega^\infty_0 \MT \Spin(4) \to \Omega^\infty_0 \MT \SO(4)
\]
induces an isomorphism in rational cohomology as well, and also on rational homotopy groups, by Serre class theory. 
\end{enumerate}

\begin{proof}[Proof of Theorem \ref{thm:hogherhomotopy}]
Item (4) of the above leaves us to prove that for a $1$-connected $4$-manifold $M$ and 
\begin{equation}\label{eqn:lastproof1}
x \in \ker (\pi_j (\Diff^+ (M)) \to \pi_j (\Homeo^+(M)))
\end{equation}
($j \geq 0$), we have 
\begin{equation}\label{eqn:lastproof2}
x \in \ker (\pi_{j} (\Diff^+ (M)) = \pi_{j+1} (B \Diff^+ (M)) \stackrel{\alpha_*}{\to} \pi_{j+1} (\Omega^\infty \MT \SO(4),\alpha(M)) \otimes \bQ). 
\end{equation}
The image of $x$ under \eqref{eqn:lastproof2} is the Madsen--Tillmann map $\alpha_{E_x}: S^{j+1} \to \Omega^\infty \MT \SO(4)$ of the smooth $M$-bundle $\pi: E_x \to S^{j+1}$ classified by $x$. 

The target group of \eqref{eqn:lastproof2} is not taken at the basepoint; when composed with the isomorphism $\pi_{j+1} (\Omega^\infty \MT \SO(4),\alpha(M)) \otimes \bQ \cong \pi_{j+1} (\Omega^\infty_0 \MT \SO(4)) \otimes \bQ$ induced by the loop space structure, the map \eqref{eqn:lastproof2} takes $x$ to the difference of the Madsen--Tillmann map of $E_x$ and the Madsen--Tillmann map of the trivial bundle $p: S^{j+1} \times M \to S^{j+1}$ (compare the arguments leading to \eqref{formula-foreqn:fromMCGtopi1mt} above). Items (1)--(3) show that we must prove 
\[
\kappa_c (E_x) - \kappa_c (S^{j+1} \times M \to S^{j+1}) \in H^* (S^{j+1};\bQ)
\]
for each $c \in H^* (B\SO(4);\bQ)$ of degree $>4$. Since $\kappa_c$ lives in positive degrees, the second summand vanishes, so we are left with $E_x$. By assumption \eqref{eqn:lastproof1}, $E_x \to S^{j+1}$ is trivial as a bundle of topological manifolds. Since Miller--Morita--Mumford classes can be defined for bundles of topological manifolds by \cite{ERW14} (this is straightforward from Novikov's theorem on the topological invariance of rational Pontrjagin classes), we obtain that $\kappa_c(E_x)=0$, for all $c \in H^{*>4}(B \SO(4);\bQ)$. 
\end{proof}

\bibliographystyle{plain}
\bibliography{literature}

\end{document}